\documentclass[a4paper,11pt]{article}
\usepackage[margin=1in]{geometry}
\usepackage{authblk} 
\usepackage{amssymb}
\usepackage{graphicx}
\usepackage{multirow}
\usepackage[title]{appendix}
\usepackage{array}
\usepackage{wrapfig}
\usepackage{amsthm}
\usepackage{amsmath}
\usepackage{graphicx}
\usepackage{verbatim }
\usepackage{caption}
\usepackage{subcaption}
\usepackage{amsmath}
\usepackage{epstopdf}
\usepackage[usenames,dvipsnames]{xcolor}
\usepackage{color}
\usepackage{booktabs}

\definecolor{myblue}{rgb}{.97, .97, .97}

\newtheorem{theorem}{Theorem}
\newtheorem{lemma}{Lemma}
\newtheorem{proposition}{Proposition}[section]
\newtheorem{corollary}{Corollary}
\newtheorem{definition}{Definition}

\newcommand{\norm}[1]{\left\Vert#1\right\Vert}

\newcommand{\tr}[1]{\mathrm{tr} #1}
\newcommand{\ip}[1]{\left \langle #1 \right \rangle }

\newcommand{\A}{ {\mathcal{A}}}

\newcommand{\OO}{\mathcal{O}}

\newcommand {\1}{\mathrm{\textbf{1}}}
\newcommand{\one}{\1}

\def \S{\mathcal{S}} 
\renewcommand{\a}{\mathbf{a}}
\renewcommand{\b}{\mathbf{b}}

\newcommand{\fnorm}[1]{\norm{#1}_F}

\DeclareMathOperator{\rank}{rank}
\DeclareMathOperator{\conv}{conv}

\newcommand{\eg}{{\it e.g.}}

\newcommand{\perm}{\Pi}

\renewcommand{\vec}{\mathrm{vec}}

\newcommand{\RR}{\mathbb{R}}

\newcommand{\N}{\mathcal{N}}

\newcommand{\ZZ}{\mathbb{Z}}

\newcommand{\Iso}{\mathrm{ISO}}
\newcommand{\Aut}{\mathrm{Aut}}
\newcommand{\CIso}{\mathrm{ISO}_\mathrm{conv}}
\newcommand{\CAut}{\mathrm{Aut}_{\mathrm{conv}}}
\newcommand{\aff}{\mathrm{aff}}

\newcommand{\affAut}{\mathrm{Aut}_{\aff}}

\newcommand{\V}{\mathcal{V}}
\newcommand{\DS}{\text{DS}}

\newcommand{\DSpp}{DS++}
\newcommand{\lambdaMin}{\lambda_{\mathrm{min}}}
\newcommand{\lambdaMax}{\lambda_{\mathrm{max}}}

\newcommand{\Sc}{S_c}

\newcommand{\showfontsize}{\f@size{} pt}

\title{Exact Recovery with Symmetries for the Doubly-Stochastic Relaxation}
\author{Nadav Dym}
\affil{Weizmann Institute of Science}

\begin{document}
\maketitle

\begin{abstract}
Graph matching or quadratic assignment, is the problem of labeling the vertices of two graphs so that they are as similar as possible. A common method for approximately solving the NP-hard graph matching problem is relaxing it to a convex optimization problem over the set of doubly stochastic (DS) matrices. Recent analysis has shown that for almost all pairs of isomorphic and asymmetric graphs, the DS relaxation succeeds in correctly retrieving the isomorphism between the graphs. Our goal in this paper is to analyze the case of symmetric isomorphic graphs. This goal is motivated by shape matching applications where the graphs of interest usually have reflective symmetry. 

For symmetric problems the graph matching problem has multiple isomorphisms and so convex relaxations admit all convex combinations of these isomorphisms as viable solutions. If the convex relaxation does not admit any additional superfluous solution we say that it is convex exact.  

 We show that convex exactness depends strongly on the symmetry group of the graphs; For a fixed symmetry group $G$, either the DS relaxation will be convex exact for almost all pairs of isomorphic graphs with symmetry group $G$, or the DS relaxation will fail for \emph{all} such pairs. We show that for reflective groups with at least one full orbit convex exactness holds almost everywhere, and provide some simple examples of non-reflective symmetry groups for which convex exactness always fails. 

When convex exactness holds, the isomorphisms of the graphs are the extreme points of the convex solution set. We suggest an efficient algorithm for retrieving an isomorphism in this case. We also show that the "convex to concave" projection method will also retrieve an isomorphism in this case, and show experimentally that this projection method as well as the standard Euclidean projection will succeed in retrieving an isomorphism for near isomorphic graphs as well.

In certain cases it is sufficient to find the centroid of the set of isomorphisms, which gives a "fuzzy encoding" of the symmetries of the shape. We show that for \emph{any} symmetry group $G$, the centroid solution can be recovered efficiently for almost all pairs of isomorphic graphs with symmetry group $G$. Additionally we show that for such isomorphic graphs  interior-point solvers will generally return the centroid solution.  
\end{abstract}

\section{Introduction}
\emph{Graph matching} and \emph{graph isomorphism} are classical problems in computer science. In this paper we will use the term \emph{graph} for a pair $(\a,A) $, where $\a=(\a_1,\ldots,\a_n) $ are the vertices of the graph, and $A$ is 
a symmetric matrix encoding the relationship between the vertices. We will also sometimes refer to $A$ alone as a graph. An \emph{isomorphism} between graphs $(\a,A) $ and $(\b,B)$ is a relabeling of the vertices of $B$ so that $A$ and the relabeled $B$ are identical. The graphs $A$ and $B$ are \emph{isomorphic} if there is an isomorphism between them. In matrix notation, an isomorphism is a permutation matrix $P $ such that $A=PBP^T $ or equivalently $AP=PB $. The problem of deciding whether two graphs are isomorphic is known as the \emph{Graph isomorphism problem} (GI). It is not known to be in P, but is also not known to be NP-hard. Recently \cite{babai2016graph} provided a quasi-polynomial time algorithm for GI. While no polynomial algorithm for the general GI problem is known, there are many families of graphs for which GI can be solved in polynomial time. One example which is relevant for this work is graphs with simple spectrum, or more generally bounded eigenvalue multiplicity \cite{babai1982isomorphism}.

The \emph{graph matching} problem is the problem of determining how close two graphs are to being isomorphic by minimizing the graph matching energy over the set of permutation matrices which we denote by $\Pi_n$:
\begin{equation}\label{e:QAP}
\min_{P \in \Pi_n} E(P)=\norm{AP-PB}_F. 
\end{equation}
 This optimization problem is also often referred to as the  Koopmans-Beckmann \emph{quadratic assignment} problem, and is usually phrased as the equivalent problem of maximizing $\tr{APB^TP^T} $. In contrast to GI whose computational status is not fully known, global minimization of quadratic assignment, and even approximation to within a constant factor, is known to be NP-hard \cite{QAPnpHardToApprox}.

 Graph matching problems have found many applications. See for example \cite{conte2004thirty} for a survey on applications of graph matching for pattern recognition. Our work is motivated by shape matching applications: Shape matching  is the problem of measuring how similar two given surfaces $\S_A,\S_B $ are. The notion of similarity between shapes is required to be invariant to shape preserving deformations  such as rigid transformations for rigid objects (\eg, chairs), and deformations which preserve geodesic distances for non-rigid objects (\eg, humans). Accordingly shape matching problems are often modeled (\eg, \cite{memoli2007use,memoli2011gromov,Justin}) as the problem of finding a mapping between two surfaces $\S_A,\S_B $ so that they are as isometric as possible.  The metric on the shapes is typically either the extrinsic Euclidean metric for rigid shapes, or the intrinsic geodesic metric for non-rigid shapes. 
 
Finding near-isometries between shapes can be phrased as a graph matching problem by selecting a finite sampling of the shapes to obtain vertices $\a,\b$ on the two shapes, and taking $A,B$ to be the distance matrices defined by the distances on the shapes, that is 
 $$A_{ij}=d_A(\a_i,\a_j), \quad B_{ij}=d_B(\b_i,\b_j) $$
 In this setting an isomorphism between $A$ and $B$ corresponds to an isometry between the sampled metric spaces.

\begin{figure}[t]
	\centering
	\includegraphics[width=\columnwidth]{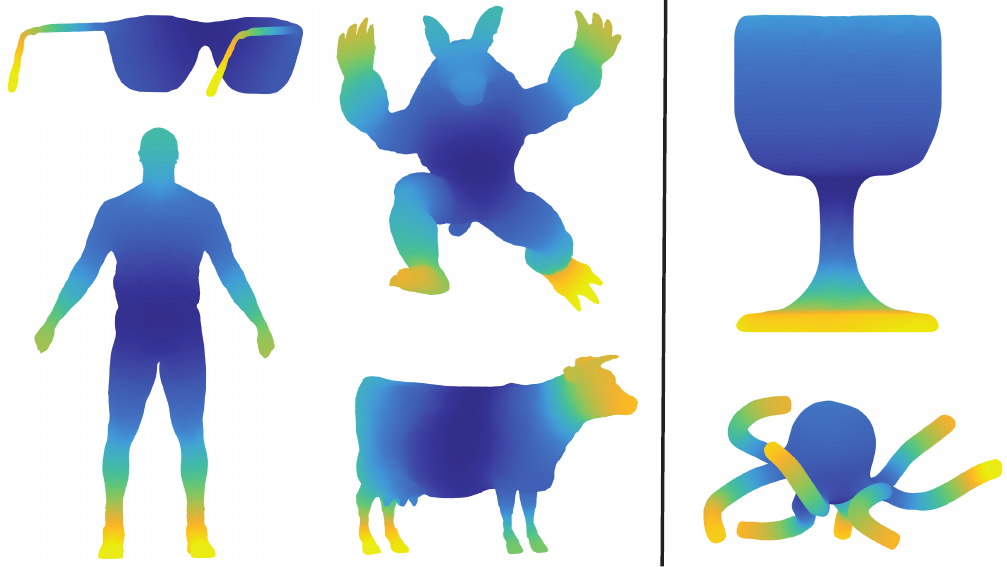}
	\vspace{+0cm}
	\caption{ \small Shapes with bilateral symmetry (left) and with non reflective symmetry groups (right) from the SHREC dataset \protect{\cite{giorgi2007shape}}. Our results suggest the DS relaxation is appropriate for solving problems with bilateral symmetry. For problems with non-reflective symmetries our results are non-conclusive (see Figure~\ref{fig:failures}).  }
	\label{fig:shapes}
	\vspace{-0.1 cm}
\end{figure}

In this work we will focus on symmetric graphs, which are very relevant for shape matching applications since most natural shapes have intrinsic symmetries- that is, intrinsic isometries from the shape to itself other than the trivial identity mapping.  Figure~\ref{fig:shapes} shows some representative shapes from the \cite{giorgi2007shape} shape matching dataset. Typically natural shapes have a symmetry group with only two elements (bilateral symmetry) as in the left hand side of Figure~\ref{fig:shapes}, but there are interesting examples with larger symmetry groups as in the right hand side of Figure~\ref{fig:shapes}.

 \paragraph{The doubly-stochastic relaxation} In this paper we focus on analyzing the doubly-stochastic (DS) relaxation for graph matching.
 For a survey on other convex relaxation and combinatorial methods which have been proposed to achieve good solutions for quadratic assignment see \cite{QAPsurvey}. 

The doubly stochastic (\DS) relaxation replaces the NP hard graph matching problem with a tractable optimization problem by relaxing the combinatorial set of permutations to its convex hull of doubly stochastic matrices:
$$\DS=\{S|\quad S \one =\one, \quad \one^TS=\one^T, \quad  S \geq 0\}, $$  
which leads to a convex quadratic program known as the \emph{DS relaxation}:
\begin{equation} \label{e:ds}
\min_{S \in \DS} E(S)=\norm{AS-SB}_F .
\end{equation}
We will refer to this optimization problem as $\DS(A,B) $. Since the \DS~ relaxation minimizes $E(\cdot)$ over a larger domain, its minimum value is a lower bound for the minimal value of the graph matching problem. As can be expected due to the hardness of the problem, the \DS~ relaxation does not generally return the global minimum or minimizer of \eqref{e:QAP} \cite{GMrisk}. In particular, \cite{fractional} characterizes all cases in which the minimum of \eqref{e:ds} is zero even when the graphs are not isomorphic. 

We will be interested in the case where $A$ and $B$ \emph{are} isomorphic. Note that in this case the global minimum of \eqref{e:ds} is zero and thus coincides with the global minimum of \eqref{e:QAP}. The interesting question is whether the \DS~relaxation succeeds in returning a minimizer which is a permutation. Clearly we do not expect this will be the case for all graphs since this would provide us with a polynomial time algorithm to solve GI. On the other hand, since there are many families of graphs for which GI is tractable, we can hope that for many instances the DS relaxation will be successful in returning a permutation solution.  The recent works of \cite{Aflalo,fiori2015spectral} show that indeed this is the case. To state their results we introduce some notation:

Let us denote the set of isomorphisms of $A,B $ by $\Iso(A,B) $. We will say that $S \in \DS$ is a \emph{convex isomorphism} if it is a member of the set
$$\CIso(A,B)=\{S \in \DS| \quad AS=SB \}. $$
The inclusion
\begin{equation}
\label{e:inclusion}
\Iso(A,B) \subseteq \CIso(A,B)
\end{equation}
is obvious. However it is possible that the \DS~relaxation will contain additional minimizers. We will say that $\DS(A,B)$ is \emph{exact} when this possibility does not occur and
\begin{equation*}
\CIso(A,B)=\Iso(A,B).
\end{equation*}
We note that the exactness property depends only on $A$: An isomorphism $P \in \Iso(A,B) $ defines a linear bijection 
$$S \mapsto SP^T $$ 
from $\CIso(A,B) $ to $\CIso(A,A)$ and from $\Iso(A,B) $ to $\Iso(A,A)$. Accordingly if $A,B$ are isomorphic, $\DS(A,B)$ is exact if and only if $\DS(A,A) $ is exact. We will refer to (convex) isomorphisms in the case $A=B $ as (convex) \emph{automorphisms}. We also denote:
 $$\Aut(A)=\Iso(A,A), \quad  \CAut(A)=\CIso(A,A), \quad  \DS(A)=\DS(A,A). $$
   We say that $A$ is an \emph{asymmetric} graph if the identity matrix is its only automorphism. Otherwise we say that $A$ is a \emph{symmetric} graph.
  A necessary condition for exactness of $\DS(A)$ is that $A $ is asymmetric. This is because if $A$ has several automorphisms then due to the inclusion \eqref{e:inclusion} and the convexity of $\CAut(A)$
  \begin{equation}\label{e:convSubseteq}
  \conv \Aut(A) \subseteq \CAut(A).
  \end{equation}
  Thus, while $A$ has a finite number of automorphisms, it has an infinite number of convex automorphisms.
Even when $A$ is asymmetric, exactness does not always occur. A simple counter example will be discussed in Section~\ref{sec:weak}. However, \cite{Aflalo} showed that for asymmetric $A$ satisfying certain weak conditions exactness will hold. Their result was later shown to hold with even weaker conditions in \cite{fiori2015spectral}. 

\paragraph{Convex exactness}
Our goal in this paper is to show that for certain kinds of symmetry groups the DS relaxation can still be successfully applied, by defining a suitable notion of \emph{convex exactness}. A similar goal has recently been achieved by \cite{PMexact} for a semi-definite programming relaxation of the Procrustes matching problem.

 We say that $\DS(A) $ is \emph{convex exact} if equality holds in \eqref{e:convSubseteq}, or equivalently if for any $B$ isomorphic to $A$,
\begin{equation}\label{e:convexExact}
\conv \Iso(A,B)=\CIso(A,B).
\end{equation}
Note that for asymmetric graphs, convex exactness and exactness coincide. When convex exactness holds an isomorphism can be extracted in a tractable manner as we will discuss in Section~\ref{sec:iso}. 

For every permutation subgroup $G \leq \Pi_n $ we define
$$\A(G)=\{A \in \S^n| \quad \Aut(A)=G \}. $$
In the asymmetric case $G=\{I_n \} $ we know that there are $A \in \A(G) $ such that $\DS(A)$ is not (convex) exact, but also that (convex) exactness often does hold for asymmetric graphs. Our goal is to give a more precise notion of this claim by showing that for almost all asymmetric graphs (convex) exactness holds. More importantly, we would like to find non-trivial groups $G$ for which  $\DS(A) $ will be convex exact for almost every $A \in \A(G) $.  
To do so we must first define a natural measure $\mu_G$ on $\A(G)$. 

We will assume that $\A(G)$ is non-empty. Permutation groups $G \leq \perm_n $, for which $\A(G) $ is empty do exist. A simple example is the cyclic group $G \leq \perm_3 $ generated by the permutation 
$$
\a_1 \mapsto \a_2, \quad \a_2 \mapsto \a_3, \quad \a_3 \mapsto \a_1.
$$
 Any $A \in \A(G)  $ satisfies
$$A_{11}=A_{22}=A_{33} \text{ and } A_{12}=A_{23}=A_{31}. $$ 
Thus, all diagonal elements of $A$ are identical and all off-diagonal elements of $A$ are identical as well since $A=A^T $. It follows that $\Iso(A)=\perm_3 $. If $\A(G)$ is non-empty we say that $G$ is a \emph{symmetry group}.

For a symmetry group $G$ we consider the vector space
$$\V(G)=\{A \in \S^n| \quad A=P^TAP \text{ for all  } P \in G \}. $$
Since $\V(G)$ is a  vector space of some dimension $d$ it has a natural notion of measure- the $d$ dimensional Hausdorff measure on $\RR^{n \times n} $ restricted to $\V(G)$, or equivalently  the push forward of the Lebesgue measure on $\RR^d$ to $\V(G) $ via a  linear isometry between the two spaces. We denote this measure by $\mu_G$.  Note that
$$\A(G)=\V(G) \setminus  \bigcup_{G \subsetneq H \leq \perm_n} \V(H). $$ 
Since by assumption $\A(G)$ is non-empty it follows that all the $\V(H) $ are strict subspaces of $\V(G) $ and therefore the complement of $\A(G)$ in $\V(G) $ has measure zero. Thus $\mu_G$ is a natural choice for a measure on $\A(G)$.  We will say that a property is \emph{generic}, or that it holds for almost every $A \in \A(G)$, if it holds for $\mu_G $ almost every $A \in \A(G)$.

We can now state our main results:

\subsection{Main results} 
\paragraph{Reflective groups} We show that convex exactness is a generic property for groups $G$ fulfilling the following two conditions:
\begin{definition}
We say that $G \leq \perm_n $ is a \emph{reflection group} if $P^2=I_n $ for all $P \in G $.
\end{definition}
Any group $G \leq \perm_n $ defines an action $(\sigma,\a_j) \mapsto \a_{\sigma(j)} $  on the set of vertices $\a$. We denote the orbit of $\a_j$ by $[\a_j]$. In general we have that $|[\a_j]| \leq |G| $.
\begin{definition}
	We say that $G$ has a full orbit if it has an orbit of length $|G|$.
\end{definition}
In shape matching applications the full orbit assumption is typically fulfilled; an orbit $[\a_i]$ will be full unless $\a_i$ is on a symmetry axis of the shape.
Under the full orbit and reflection group assumption, we prove:
\begin{theorem}\label{th:main}
Assume $G \leq \perm_n $ is a reflection group with a full orbit. Then the DS relaxation is convex exact with respect to almost all $A \in \A(G) $.
\end{theorem} 
As a result we obtain that convex exactness is a generic property for the simplest but, in the context of shape matching applications, most important, symmetry groups:
\begin{corollary}\label{cor:Z2}
If $G \cong \ZZ_2 $ then the DS relaxation is convex exact with respect to almost all $A \in \A(G) $.
\end{corollary}

\paragraph{General groups}
For general groups we provide a "zero-one probability" result:
\begin{theorem}\label{th:allOrNothing}
	For any symmetry group $G \leq \perm_n $ one of the following holds:
	\begin{enumerate}
		\item The DS relaxation is convex exact with respect to almost every $A \in \A(G)$.
		\item The DS relaxation is \textbf{not} convex exact for any $A \in \A(G) $.
	\end{enumerate}
\end{theorem}

\begin{wrapfigure}[19]{R}{0.5\textwidth}
	\centering
	\includegraphics[width=0.5\textwidth]{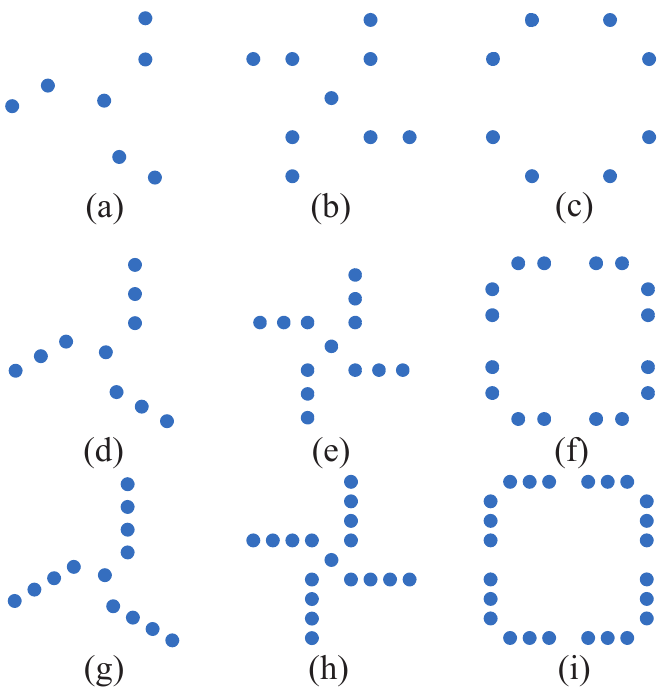}
	\vspace{-0.7 cm}
	\caption{\small Non-reflective symmetry groups.}
	\label{fig:failures}
	\vspace{0 cm}
\end{wrapfigure}
The proof of Theorem~\ref{th:allOrNothing} is constructive in the sense that it enables checking which of the two mutually exclusive alternatives described in the theorem hold for a given symmetry group $G$. By using this strategy we can establish that there are quite simple non-reflective symmetry groups for which convex exactness fails. Figure~\ref{fig:failures} shows nine groups $G_i$ represented by nine shapes whose symmetry group is $G_i$. For the first three groups (a)-(c) we found that convex exactness does not hold for any $A \in \A(G) $, while for the remaining groups convex exactness does hold for almost all $A \in \A(G) $. Note that all groups in the first column are isomorphic to $\ZZ_3$, all groups in the second column are isomorphic to $\ZZ_4$, and all groups in the last column are isomorphic to the dihedral group $D_4$. Thus we see that while convex exactness is a generic property for any $G$ isomorphic to $\ZZ_2$, in general different permutation groups can behave very differently with respect to the DS relaxation even if they are isomorphic in the sense of group theory.

\paragraph{Additional results: Permutation solutions and centroid solution} 
Since for symmetric problems the DS relaxation has an infinite number of convex isomorphisms, the question of achieving an "interesting" convex isomorphism arises. 
Naturally we would like to achieve a convex isomorphism which is a permutation. In the case of convex exactness this reduces to the problem of finding an extreme point (a "corner") of the set of convex isomorphisms, which is known to be a tractable problem. In Section~\ref{sec:iso} we describe two known methods to obtain extreme points. Additionally we provide a much faster algorithm for achieving all isomorphisms between $A$ and $B$. This algorithm is valid  for almost all graphs whose symmetry group $G$ satisfy the conditions of Theorem~\ref{th:main}.

A disadvantage of the methods mentioned above for finding permutation solutions is that they are constructed for perfectly isomorphic problems and are not suited for near isomorphic problems and are not used in practice. Instead, permutations are typically obtained using the $L_2$ projection or the more accurate, but more expensive, "convex  to concave" projection. We prove  that when convex exactness holds, the convex to concave projection is able to return an isomorphism. For  symmetric problems with a small amount of noise, we show experimentally that both projection methods are generally able to retrieve an isomorphism, and the convex to concave method is often able to retrieve an isomorphism for higher noise levels as well.    

An alternative "interesting" convex isomorphism which is easier to find than "corners" is the "centroid" of the set of isomorphisms:
\begin{equation}
\label{e:maxEntropy}
\Sc=\frac{1}{|G|}\sum_{P \in G}P, 
\end{equation}
where $G $ is the set of isomorphisms between $A$ and $B$. As advocated in \cite{softMaps}, finding $\Sc$ in the case of symmetric problems can potentially be useful as it gives an "encoding" of all isomorphisms of $A,B$. In Section~\ref{sec:centroid} we show that  the centroid solution is easier to find than corner solutions: In fact, for any symmetry group $G$, and almost every pair of isomorphic graphs $A,B \in \Iso(G) $, the centroid solution can be achieved (almost always) for \emph{any} symmetry group. Additionally we show that for such $A,B $  penalty based optimization methods will converge to $\Sc$ when solving $\DS(A,B) $. 

\begin{wrapfigure}[10]{R}{0.5\textwidth}\vspace{-0.4cm}
	\centering
	\includegraphics[width=0.5\textwidth]{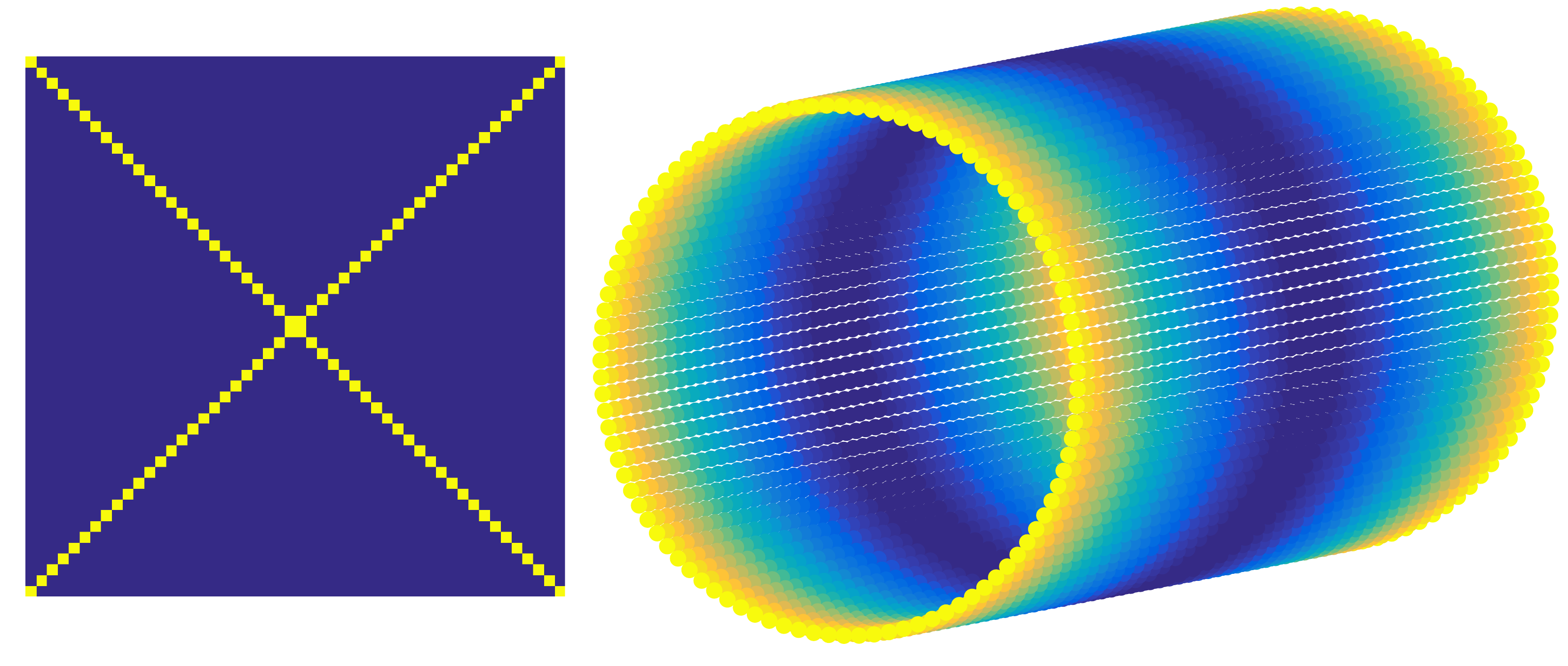}
	\caption{\small The centroid solution.}
	\label{fig:centroid}
	\vspace{+0cm}
	\vspace{-0.1 cm}
\end{wrapfigure}
An illustration of the centroid solution is shown in Figure~\ref{fig:centroid}, for the problem of mapping a cylinder to itself, using as $A=B$ the Euclidean distance matrix of the cylinder. The right part of the figure shows the cylinder, colored so that points in the same orbit of the symmetry group of the cylinder share the same color. The left part of the figure shows the matrix $\Sc$.  Each yellow square in the left figure is  a submatrix whose indices  correspond to a circular section of  the cylinder.   It can be seen that the centroid solution assigns each point of the cylinder with equal probability to any other point in its orbit. We note that the centroid solution always has this property. Therefore different symmetry groups which have identical orbits will have the same centroid solution, and so the symmetry group cannot generally be reconstructed from the centroid solution.

The remainder of the paper is organized as follows: In Section~\ref{sec:weak} we define the notion of weak exactness which will be useful for the proofs presented later on. In Section~\ref{sec:exact}
we prove convex exactness for reflective groups (Theorem~\ref{th:main}). In Section~\ref{sec:allOrNothing} we prove our "zero-one probability result" (Theorem~\ref{th:allOrNothing}) and explain how to check which of the two alternatives described in the theorem apply for a given group. In Section~\ref{sec:centroid} we discuss the issue of retrieving the centroid solutions and finally in Section~\ref{sec:iso} we discuss the issue of retrieving isomorphisms in the case that convex exactness holds.


\section{Weak exactness}\label{sec:weak}
An important tool for the proofs we present later on is the concept of \emph{weak exactness} which we will now define: For any set $G$ of permutation matrices we define
\begin{equation*}
\N(G)=\{T \in \RR^{n \times n}| \quad T_{ij}=0 \text{ if } P_{ij}=0 \text{ for all } P \in G \}. 
\end{equation*}
We say that $\DS(A)$ is weakly exact if all convex automorphisms of $A$ are in $\N(\Aut(A)) $.
Less formally, this means that the $i,j$ coordinate of a convex automorphism $S$ can be positive only if there is an automorphism taking $\a_i$ to $\a_j$.
If $\DS(A)$ is weakly exact and $B$ is isomorphic to $A$ then all convex isomorphisms of $A,B$ are in $\N(\Iso(A,B)) $. 
Weak exactness is guaranteed with full probability for any symmetry group $G$. We show this using the vector
\begin{equation*}
s(A)=A \one.
\end{equation*}
The vector $s(A)$ is invariant under automorphisms, meaning that if $\a_j \in [\a_i] $ then $s_i(A)=s_j(A) $. We say that $s(A)$ is \emph{discriminative} if for any $i,j$ such that $\a_j \not \in [\a_i] $, we have $s_i(A) \neq s_j(A) $. We prove
\begin{theorem}\label{th:weakExactness}
	Let $G$ be \emph{any} symmetry group. Then
	\begin{enumerate}
		\item If $s(A)$ is discriminative then $\DS(A)$ is weakly exact.
		\item For almost every $A \in \A(G) $, the vector $s(A)$ is discriminative.
	\end{enumerate}
\end{theorem}

We prove Theorem~\ref{th:weakExactness}. We first prove that if $s(A)$ is discriminative then $DS(A)$ is weakly exact. If $S$ is a convex isomorphism, then $s=s(A)$ is fixed by $S$ because
$$Ss=SA\one=AS\one=A\one=s. $$
Write $S$ as a convex combination of permutations $S=\sum_k \theta_k P(k) $. Using the fact that the operator norm of a permutation is one and the Cauchy-Schwartz inequality we obtain 
$$\norm{s}^2=\ip{Ss,s}=\sum_k \theta_k \ip{P(k)s,s} \leq^{(*)} \sum_k \theta_k \norm{P(k)}\norm{s}^2=\norm{s}^2. $$
so $(*)$ is an equality, implying that $P(k)s=s $ for all $k$. Now if $\a_j \not \in [\a_i] $ then $P_{ij}(k)=0 $ for all $k$ and therefore $S_{ij}=0 $. Thus we have proven that $\DS(A)$ is weakly exact when $s(A)$ is discriminative.

We now show that discriminativeness is a generic property. It is sufficient to show that for almost every $A \in \V(G) $ the claim holds since $\A(G) $ is a subset of $\V(G)$. Note that $s(A) $ is discriminative unless there are some $(i,j) $ such that  $\a_j$ is  not in $[\a_i] $ but $A$ is in the vector space 
	$$\V_{i,j}(G)=\{A \in \V(G)| s_i(A)=s_j(A) \}. $$
Thus it is sufficient to show that all these spaces are strict subspaces of $\V(G)$, which we accomplish by finding
 a member $\bar A \in \V(G)$ for which $s(\bar A)$ is discriminative. 
 
 To construct $\bar A$ let $I_1,I_2,\ldots,I_k$ be the partition of the vertices $\a$ induced by the action of $G$. For all $r \leq k $ and  $\a_i,\a_j \in I_r $ we set
 $$\bar A_{ij}=r|I_r |^{-1}. $$   	
If $\a_j \not \in [\a_i] $ we set $\bar A_{ij}=0$. The constructed graph $\bar A$ is a member of $\V(G)$, and the vector $s(\bar A) $ is discriminative since for all $r \leq k$ and $i \in I_r $ we have $s_i(\bar A)=r $.  This concludes the proof of Theorem~\ref{th:weakExactness}.

\paragraph{Counter example}
While weak exactness is guaranteed almost everywhere, it can still fail in very simple examples. Such examples can be constructed using the fact that if $\one $ is an eigenvector of $A$ then $\frac{1}{n}\one^T\one $ is always a valid convex automorphism. For example the graph
$$A_0=\begin{bmatrix}
6 & 1 & 2\\
1 & 5 & 3\\
2 & 3 & 4
\end{bmatrix} $$
is asymmetric, but satisfies $A_0 \one =\lambda \one $ for $\lambda=9$ and so $\frac{1}{n}\one^T\one $ is a convex isomorphism, and so weak exactness, and certainly exactness, does not hold. 

One method for overcoming such counter examples is adding a linear term to the graph matching energy penalizing for correspondences which do not respect isomorphism-invariants. For example, For each vertex $\a_i$ we can define $L(i) $ to be the sorted values of the $i$-th row of the graph. Clearly if $\a_j \in [\a_i] $ then $L(i)=L(j) $ so $L$ is an isomorphism invariant. Since in our example $L(i), i=1,2,3 $ are all distinct, the only zero-energy solution of the modified relaxation
\begin{equation}\label{e:invariants}
\min_{S \in \DS} \fnorm{SA_0-A_0S}^2+\sum_{ij}S_{ij}\norm{L(i)-L(j)}
\end{equation}
is the identity matrix.

\section{Convex exactness for reflective groups}\label{sec:exact}
Our goal in this section is proving convex exactness holds generically for reflective groups with a full orbit (Theorem~\ref{th:main}). We break up the proof of the theorem into two parts: The first part establishes sufficient  conditions which guarantee exact recovery, and the second part proves these sufficient conditions hold generically if $G$ is reflective and has a full orbit.

Fix some $G$ and $A \in \A(G) $. As in the previous section let $I_1,I_2,\ldots,I_k$ be the partition of the vertices $\a$ induced by the action of $G$, and denote $n_j=|I_j| $.
Let $S$ be a convex automorphism of $A$. Denote by $S_{ij} $ and $A_{ij} $ the submatrices of $S,A$ corresponding to the indices  $I_i \times I_j $. If $A$ is weakly exact then $S_{ij}=0 $ whenever $i \neq j $ and so  the equation $AS=SA $ takes the form 
\begin{equation}\label{e:blockEq}
A_{ij}S_{jj}=S_{ii}A_{ij}, \quad  \text{ for all } 1\leq i,j \leq k.
\end{equation}

\subsection{Sufficient conditions for convex exactness}
\begin{proposition}\label{prop:deterministic}
		Let $A$ be a graph. If $\exists r, 1 \leq r \leq k $ such that
	\begin{enumerate}
		\item The vector $s(A)$ is discriminative.
		\item $\text{rank}(A_{rj})=n_j $  for all $j \neq r$.
		\item $A_{rr}$ has simple spectrum.
	\end{enumerate}
	then $\DS(A)$ is convex exact. 
\end{proposition}

\begin{proof}[Proof of Proposition~\ref{prop:deterministic}]
	The first condition guarantees weak exactness, and thus that all convex isomorphisms will satisfy \eqref{e:blockEq}. Setting $i=r$ in this equation we obtain
	$$A_{rj}S_{jj}=S_{rr}A_{rj}, \quad \text{ for all } j $$
	and so by the second condition $S_{jj} $ is determined uniquely by $S_{rr}$. It follows that the restriction of the linear map
	$$S \mapsto S_{rr} $$
	to $\CAut(A) $ is injective. Therefore it is sufficient to show that $S_{rr} $ is a convex combination of the permutation matrices $P_{rr} $ obtained by restricting the automorphisms $P \in \Aut(A)$ to $I_r \times I_r $. By taking $i=r,j=r $ in \eqref{e:blockEq} we see that $S_{rr}$ is a  convex automorphism of the subgraph $A_{rr}$, and  the group
	$$H=\{P_{rr}| \quad P \in \Aut(A) \} $$
	is a subgroup of $\Aut(A_{rr}) $ which acts transitively on $I_r $. By the third assumption $A_{rr} $ has simple spectrum. Thus to show $A_{rr}$ is a convex combination of elements of $H$ it is sufficient to prove
	\begin{lemma}
	If $A\in \S^n $ is a graph with simple spectrum, and $H\leq \Aut(A)$ acts transitively on the vertices $\a$, then $H=\Aut(A) $ and $\DS(A) $ is convex exact.
	\end{lemma}
    We now conclude the proof of the proposition by proving the lemma. In this proof $B_j$ denotes the $j$-th column of the matrix $B$, and $B_{i \star} $ denotes the $i$-th row of $B$.
    
    To prove the lemma it is sufficient to show that
    \begin{equation}\label{e:sufficient}
    \CAut(A) \subseteq \conv H
    \end{equation}
    because this implies that 
    $$\conv \Aut(A) \subseteq \CAut(A) \subseteq \conv H \subseteq \conv \Aut(A) $$
    which proves that $\DS(A)$ is convex exact. Additionally all automorphisms $P \in \Aut(A)$ are in $\conv H$, and since permutations are extreme points of $\DS$ this can only occur if $P \in H$, and so $H=\Aut(A) $.
    
     We prove \eqref{e:sufficient} using   an argument from \cite{PMexact}; If  $S$ is a convex automorphism of $A$ and $v$ is an eigenvector of $A$ with eigenvalue $\lambda$, Then
     $$ASv=SAv=\lambda Sv $$
     so either $Sv=0 $ or $Sv$ is an eigenvector of $A$ with eigenvalue $\lambda$. Since $A$ has simple spectrum it follows that $Sv=\alpha v $ for some $\alpha \in \RR $. If $S=P $ is an automorphism of $A$ then $\alpha \in \{-1,1\} $.
     
     Let $S$ be a convex automorphism of $A$. We want to show that $S \in \conv H $. Since $H$  acts transitively on $\a$, there are permutation matrices $P(1),\ldots,P(n) \in H $ such that for any vector $w\in \RR^n$
     $$w_k=(P(k)w)_1. $$
      Denote by $V$ the matrix whose columns are the eigenvectors of $A$. Then there is a diagonal matrix $D$ and diagonal matrices $D(1),\ldots,D(n) $ such that 
     \begin{equation} \label{e:SVVD}
     D=V^TSV \text{ and } D(k)=V^TP(k)V, \text { for all } 1 \leq k \leq n.
     \end{equation}  
	Note that $S$ is a convex combination of $P(1),\ldots,P(n) $ if and only if $D $ is a convex combination of $D(1),\ldots,D(n) $. If $v$ is the $j$-th eigenvector of $A$ then 
		$$|v_k|=|(P(k)v)_1|=|D_{jj}(k)v_1|=|v_1| $$
		so in particular $v$ has no zero coordinates.
		
	From \eqref{e:SVVD} we obtain
		\begin{equation*}
		(VD)_{1 \star}=( SV)_{1 \star}=\sum_{k=1}^n S_{1 k}V_{k \star}=\sum_{k=1}^n S_{1 k}(P(k)V)_{1 \star}=\sum_{k=1}^n  S_{1 k}(VD(k))_{1 \star} \;
		\end{equation*} 
		and therefore
		$$
		V_{1 \star}\left(D-\sum_{k=1}^n  S_{1 k}(D(k))\right)=0.
		$$
		Since all entries of $V$ are non-zero the only diagonal matrix solving the equation above is the zero matrix. Thus we obtain $D $ as a convex combination of $D(k) $:
		$$D=\sum_{k=1}^n  S_{1 k}(D(k)). $$     
	\end{proof}
\subsection{Genericity of the sufficient conditions}
In this subsection we  prove the sufficient conditions of Proposition~\ref{prop:deterministic} hold generically if $G$ is reflective and has a full orbit. The first condition was proved to hold generically for any symmetry group $G$ in Theorem~\ref{th:weakExactness}. We choose the $r$ appearing in the last two conditions of of Proposition~\ref{prop:deterministic} such that $I_r$ is a full orbit, or equivalently $|I_r|=|G| $. We begin with some preliminaries.  	

\paragraph{Preliminaries}
Recall that for a symmetry group $G$ of dimension $d$, the measure $\mu_G$ can be defined as the restriction to $\V(G)$ of the $d$-dimensional Hausdorff measure $H^d $ on $\RR^{n \times n}$.  
We cite some basic properties of the Hausdorff measure and dimension from chapter 2 in \cite{falconer2004fractal} which will be helpful for the proof of Lemma~\ref{l:simple}.
\begin{enumerate}
	\item If $C $ has Hausdorff dimension $k$ and $s>k$, then $H^s(C)=0$.
	\item If $M \subseteq \RR^n$ is a submanifold of dimension $d$, then its Hausdorff dimension is $d$ as well.
	\item If $B=\cup_{i \in I} B_i$ and $I$ is countable,  then $\dim B=\sup_{i \in I} \dim B_i $.
	\item If $B \subseteq C $ then $\dim(B) \leq \dim(C) $.
	\item If $B \subseteq \RR^m $ and $f:B \to \RR^n $ is Lipschitz, then $\dim f(B)\leq \dim(B) $. 
	
	 An immediate consequence is that the latter inequality holds if $f$ is a $C^1$ function defined on all of $\RR^m $. To see this denote $B_k=B \cap \{x| \quad \|x\|\leq k \} $ and note that the restriction of $f$ to $B_k $ is Lipschitz. Therefore 
	\begin{equation} \label{e:dimf}
	\dim f(B)=\dim \cup_k f(B_k)=\sup_k \dim f(B_k) \leq \sup_k \dim B_k \leq \dim B 
	\end{equation}
\end{enumerate}

For Lemma~\ref{lem:fullRank} we will need the following simple lemma. We include a proof for completeness:
\begin{lemma}\label{l:variety}
	If $p(x)$ is a non-zero multivariate polynomial $p:\RR^d \to \RR $, then the set $\{x| \;  p(x)=0\}$ has Lebesgue measure zero.
\end{lemma}
\begin{proof}
	By induction. For $d=1$ the claim is obvious. We assume the claim holds for $d-1$ and show it holds for $d$. Rewrite $p$ as
	$$p(x)=\sum_{j=1}^n q_j(x_1,\ldots,x_{d-1})x_d^j. $$
	By the induction hypothesis the set
	$$C=\{(x_1,\ldots,x_{d-1})| \quad \forall j , \quad   q_j(x_1,\ldots,x_{d-1})=0\} $$
	has measure zero in $\RR^{d-1} $. For any fixed $(x_1,\ldots,x_{d-1}) $ in the complement of $C$, $p(x)$ is a univariate non-zero polynomial and has zeros in a (finite) subset of $\RR$ of measure zero. Using Fubini's theorem this implies that the set $\{x| \;  p(x)=0\}$ has Lebesgue measure zero.
\end{proof}
\paragraph{Proof of genericity}
	
A graph $A$ is in $\V(G)$ if it is symmetric and \eqref{e:blockEq} is satisfied when $S$ is replaced with all permutations $P \in G $. This means that $\V(G)=\oplus_{i \geq j} \V_{ij}(G) $ where
$$\V_{jj}(G)=\{A_{jj} \in \S(n_j)| \quad  A_{jj}P_{jj}=P_{jj}A_{jj} \text{ for all } P \in G \}. $$
and for $i>j$ 
$$\V_{ij}(G)=\{ A_{ij} \in \RR^{n_i \times n_j}| \quad A_{ij}P_{jj}=P_{ii}A_{ij}, \quad \text{ and } \quad A_{ij}^TP_{ii}=P_{jj}A_{ij}^T.
\} $$

Thus to prove the second condition is generic it is sufficient to show that almost every $A_{rr} \in \V_{rr}(G) $ has simple spectrum, and that to prove the third condition is generic we need to show that almost every $A_{rj} \in \V_{rj}(G) $ has full rank. Thus the second condition follows by setting $A=A_{rj} $ and $\V=\V_{rj}$ in the following Lemma: 
\begin{lemma}\label{l:simple}
If $G$ is reflective then almost all $A \in \V(G)$ has simple spectrum.	
\end{lemma}
\begin{proof}

Since $G$ is reflective all $P \in G $ satisfy
 $$P=P^{-1}=P^T. $$

Members $P,Q \in G $ commute because
$$PQ=(PQ)^T=Q^TP^T=QP. $$
Thus $G$ can be diagonalized simultaneously, and so we can partition $\RR^n$ into a direct sum of eigenspaces $\RR^n=\oplus_{i=1}^{\ell} W_i $. We denote the dimension of each eigenspace by $d_i$. Select for each subspace $W_i$ a matrix $V_i \in \RR^{n \times d_i} $ whose columns form an orthogonal eigenbasis of $W_i$, and denote
$$V=[V_1,\ldots,V_\ell]. $$

A graph $A$ is  in $ \V(G) $ if and only if it is symmetric and it commutes with the members of $G$. This in turn occurs if and only if $A$ and all members of $G$ can be diagonalized simultaneously, and so there are symmetric matrices $\bar A_i, \quad i=1,\ldots,\ell $ such that 
$$V^TAV= 
\left[ \begin{array}{cccc}
\bar A_1 &           &            &         \\
         & \bar A_2  &            &         \\
         &           & \ddots     &         \\
         &           &            & \bar A_{\ell}
\end{array} \right]       
$$ 
It follows that $\V(G) $ can be identified with $\oplus_{i=1}^\ell S(d_i) $, and is thus of dimension
$$\dim(\V(G))=\sum_{i=1}^\ell \frac{d_i^2+d_i}{2} $$ 

For $(U_i,\lambda_i) \in \OO(d_i) \times \RR^{d_i} $ we define $D(\lambda_i)$ to be the diagonal matrix whose diagonal entries are $\lambda_i$, and define
$$\bar A_i(U_i,\lambda_i)=U_iD(\lambda_i)U_i^T. $$
 Now consider the function $f:\prod_{i=1}^\ell \left( \OO(d_i) \times \RR^{d_i} \right) \to \RR^{n \times n} $ defined by
$$f(U_1,\lambda_1,\ldots,U_\ell,\lambda_\ell)= 
V \left[ \begin{array}{cccc}
\bar A_1(U_1,\lambda_1) &           &            &         \\
& \bar A_2(U_2,\lambda_2)  &            &         \\
&           & \ddots     &         \\
&           &            & \bar A_{\ell}(U_\ell,\lambda_\ell)
\end{array} \right] V^T    $$
The image of $f$ is precisely $\V(G) $. Moreover the dimension of the domain of $f$ is
\begin{equation*}
\dim \text{dom}(f)=\sum_{i=1}^\ell \frac{d_i^2-d_i}{2}+d_i=\dim \V(G).
\end{equation*}
The complement of the set of graphs $A \in \V(G) $ with simple spectrum is a union of sets of the form $f(E_{qr}) $ where
$$E_{qr}=\{(U_i,\lambda_i)_{i=1}^\ell| \quad \text{ for } \bar \lambda=(\lambda_1,\ldots,\lambda_\ell), \quad  \bar \lambda_q=\bar \lambda_r   \}.$$ 
Since the dimension of each such set is strictly smaller than the dimension of the domain we obtain:
$$\dim f(E_{qr}) \leq^{\eqref{e:dimf}} \dim E_{qr} \leq  \dim \text{dom}(f)-1=\dim \V(G)-1$$
and so the complement of the set of graphs $A \in \V(G) $ with simple spectrum is dimension deficient and thus has zero Hausdorff measure. 
\end{proof}
We now prove the third condition holds generically. 
\begin{lemma}\label{lem:fullRank}
If $G$ has a full orbit $I_r $ then almost every $A_{rj} \in \V_{rj}(G) $ has full rank.
\end{lemma}
\begin{proof}
	In this proof we denote members of the vector space $\V_{rj}(G)$ by $\bar  A$. We identify this vector space  with $\RR^\ell $ for some $\ell$ via a linear isomorphism $\bar A: \RR^{\ell} \to \V_{rj}(G) $, and define a multivariate polynomial $p: \RR^{\ell} \to \RR$ by
	$$p(x)=\text{det} \left(\bar{A}^T(x)\bar{A}(x) \right) $$
	Note that $p(x)=0$ if and only if $\bar A(x)$ has full rank. Thus due to Lemma~\ref{l:variety} it is sufficient to show that $p$ isn't identically zero, or equivalently, establish the existence of a full rank matrix in $\V_{rj}(G)$. We now construct such a matrix which we will denote by $\hat A$.
	
	Note that $\hat A \in \V_{rj}$ if and only if
	\begin{equation} \label{e:constraints}
	\hat A_{sq}=\hat A_{\sigma(s) \sigma (q)}, \text{ for all } s \in I_r, q \in I_j, \sigma \in G.
	\end{equation}
	This means that the values of $\hat A$ are required to be constant along the orbits of the action of the group $G$ on $I_r \times I_j $ defined by
	$$(\sigma,(s,q)) \mapsto (\sigma(s),\sigma(q)). $$
	We choose some orbit $[(s,q)] $ of this action and define $\hat A$ by the requirement that $\hat A_{ij}=1 $ if $(i,j) $ are member of this orbit, and otherwise $\hat A_{ij}=0 $. By construction $\hat A \in \V_{rj}(G)$ and it remains to verify that it has full rank. The orbit $[(s,q)]$ has $n_r=|G|$ elements $(s_1,q_1),\ldots,(s_{n_r},q_{n_r}) $ where the $s_i$ are all distinct, and we can order the orbit so that the first $n_j$ elements of the sequence $q_i $ are distinct as well. Thus 
	 \begin{equation*}
	 (\hat A_{s_i,q_\ell})_{i,\ell=1}^{n_j}=I_{n_j}
	 \end{equation*}
	 and so $\rank \hat A=n_j $.
	 \end{proof}
\section{Almost all or nothing}\label{sec:allOrNothing}
In this section we prove Theorem~\ref{th:allOrNothing}; we show that for any symmetry group $G$, either the DS relaxation is never convex exact for any $A \in \A(G)$, or the DS relaxation is convex exact for almost every $A \in \A(G) $. We then explain how generic convex exactness can be established/refuted for a given group $G$.

Our proof uses another notion of exactness which we will call \emph{affine exactness}:
The affine automorphisms of a graph $A$ are the members of the affine set
$$\affAut(A)=\{S| \quad S \one=\one, \quad \one^TS=\one^T,S \in \N(\Aut(A)), AS=SA  \}. $$
We note that affine automorphisms and convex automorphisms differ in two aspects: On the one hand the entries of convex automorphisms are required to be non-negative while the entries of affine automorphisms are not.  On the other hand,  affine automorphisms must be members of $\N(\Aut(A)) $, a requirement we do not impose on convex automorphisms (although by Theorem~\ref{th:weakExactness} convex automorphisms will "usually" satisfy this property).   

We say that affine exactness holds at $A$ if 
$$\affAut(A)=\aff \Aut(A). $$

 We begin by establishing a  connection between affine exactness and convex exactness:

\begin{proposition}
For any graph $A\in \A(G)$, convex exactness holds at $A$ if and only if affine exactness holds at $A$ and 
\begin{equation} \label{e:weird_condition}
\aff \Aut(A) \cap \{S \geq 0 \}=\conv \Aut(A).
\end{equation}
\end{proposition}
Note that the RHS of \eqref{e:weird_condition} is always contained in the LHS.
\begin{proof}
	If affine exactness and \eqref{e:weird_condition} hold then
	$$\conv \Aut(A) \subseteq \CAut(A) \subseteq \affAut(A) \cap \{S \geq 0 \}=\aff \Aut (A) \cap \{S \geq 0 \}= \conv \Aut(A). $$
	
	so convex exactness holds as well. 
	
	Now assume convex exactness holds. Then \eqref{e:weird_condition} holds since
	$$\conv \Aut(A) \subseteq \aff \Aut(A) \cap \{S \geq 0 \} \subseteq \CAut(A)=\conv \Aut(A). $$
	
	 To show affine exactness holds we choose some $ S \in \affAut(A) $ and show that $S \in \aff \Aut(A) $.  The "centroid" convex automorphism $\Sc$  is non-zero in all coordinates except for coordinates $(i,j) $ which satisfy $P_{ij}=0 $ for all automorphisms $P$. Since at such coordinates $S$ is also zero,  there  is some $\epsilon>0 $ such that
	$$S_1=(1-\epsilon)\Sc+\epsilon S $$
	is doubly stochastic. $S_1$ is also an affine automorphism, and thus is a convex automorphism. By assumption $S_1$ is a convex combination of members of $\Aut(A) $. In particular $S_1,\Sc \in \aff \Aut(A) $ and therefore so is
	$$S=\frac{1}{\epsilon}\Sc+(1-\frac{1}{\epsilon})S_1. $$
\end{proof}

Note that the condition \eqref{e:weird_condition} depends only on $G$ and not on a specific choice of $A \in \A(G)$. Thus if the condition does not hold then $\DS(A)$ will not be convex exact for any $A \in \A(G) $. If $G$ is such that the condition does hold then convex exactness for specific $A \in \A(G)$ is equivalent to affine exactness. Thus Theorem~\ref{th:allOrNothing} follows from the following proposition: 
\begin{proposition}\label{prop:affine}
Let $G$ be a symmetry group. Then
 either affine recovery holds for almost every $A \in \A(G) $,
 Or affine recovery fails for all $A \in \A(G) $.
\end{proposition}

\begin{proof}[Proof of Proposition~\ref{prop:affine}]
As in the proof of Lemma~\ref{lem:fullRank} we show the set of $A \in \A(G) $ for which affine recovery fails is a null set of a suitable multivariate polynomial.

For given $A \in \A(G) $, the affine automorphisms of $A$ are the matrices $S$ satisfying the affine equations defining $\affAut(A) $. Denoting   the  map which identifies $n \times n $ matrices $S $ with $n^2 \times 1 $ vectors by $\vec(\cdot)$, these equations can be written in the form
\begin{equation}\label{e:vectorized}
F(A) \vec(S)=b 
\end{equation}
where $F(A) $ depends linearly on $A$. Since $A \in \A(G) $, all members of $\vec(G) $ are solutions of \eqref{e:vectorized}. Thus  the kernel of $F(A)$ always includes
$$W=\mathrm{span}\left( \vec (G-I_n) \right) $$ 
and affine exactness holds iff $\mathrm{Ker}(F(A))=W $. Let 
$$U=[U_0, U_1] $$
be a unitary matrix, such that $U_0 $ forms an orthonormal basis of $W$. Then affine exactness holds iff  $F(A)U_1$ has full rank. Now pick some linear isometry $x \mapsto \bar A(x) $ from $\RR^\ell $ to the vector space $\V(G) $. affine exactness holds at $\bar A(x) $ iff $x$ is not a zero of the multivariate polynomial
$$p(x)=\det\left(U_1^TF^T(\bar A(x))F(\bar A(x)U_1) \right).$$ 
This concludes the proof of the proposition, due to Lemma~\ref{l:variety}.
\end{proof}

\paragraph{Checking exactness for given groups} We now explain how we check whether convex exactness holds generically for the groups $G_i, i=1,\ldots,9 $ defined by the shapes in Figure~\ref{fig:failures}. We first note that condition~\eqref{e:weird_condition} holds for these groups. This is because $G_i $ all contain a full orbit. If $G$ has a full orbit $[\a_k] $ and $S$ is an affine combination of members of $G$, then the coefficients of the affine combination are just the values of the column $S_k$. In particular if $S$ is doubly stochastic then the affine combination is in fact a convex combination since $S_k$ is a probability vector.

Since \eqref{e:weird_condition} holds we have generic convex exactness for $G_i$ if and only if the polynomial $p=p(G_i)$ is a non-zero polynomial. This can be checked either by computing the polynomial symbolically or by evaluating it on random input. We used the latter method: For each of the groups $G_i$ we generated $100$ random graphs in $\V(G_i)$ and evaluated the polynomial on these graphs. For groups $G_1,\ldots,G_3 $ all $100$ evaluations of the polynomial were zero, and for the remaining graphs the polynomial was found to be non-zero at all $100$ evaluated points. These results are summarized in the first column of Table~\ref{tab:failures}.  

\section{Centroid}\label{sec:centroid} 
 Recall that the centroid solution $\Sc$ is the matrix obtained by averaging over all members of $G$ as defined in \ref{e:maxEntropy}. In this section we show that for any symmetry group $G$ and almost every $\A \in \A(G)$ the centroid solution $\Sc$ can be recovered efficiently. We also show that $\Sc$ is the solution which will be obtained by interior-point methods when solving $\DS(A)$. 

We begin by giving an explicit construction of $\Sc$. Let us denote as before the equivalence classes of the action of $G$ on $\a$ by $I_1,\ldots,I_k $. Assume that the vertices are arranged so that 
$$I_1=\{\a_1,\ldots,\a_{m_1} \}, \quad I_2=\{\a_{m_1+1},\ldots,\a_{m_2} \},\quad \ldots, \quad I_k=\{\a_{m_{k-1}+1},\ldots,\a_{m_k} \}. $$
Set $n_j=|I_j|, j=1,\ldots,k $, and for any integer $p$ let $J_p \in \RR^{p \times p}$ be the constant matrix whose entries are all $p^{-1} $. Note that $\Sc$ is in $\N(G)$ and is invariant under multiplication from the left or right by elements of $G$. The only doubly stochastic matrix satisfying these properties is
\begin{equation*}
S_0=\begin{bmatrix}
J_{n_1} &         &        &     \\
        & J_{n_2} &        &     \\
        &         & \ddots &     \\
        &         &        & J_{n_k}
\end{bmatrix} 
\end{equation*}
and therefore $\Sc=S_0 $. Recall that for any symmetry group $G$ and almost every $A \in \A(G) $ the vector $s(A)$ is discriminative. In this case the centroid solution $\Sc=S_0$ can be easily computed without even solving the DS relaxation: We first construct a matrix $S$ by setting $S_{ij}=1 $ if $s_i(A)=s_j(A) $ and $S_{ij}=0 $ otherwise. We can then obtain $\Sc$ by normalizing the rows of $S$.

\paragraph{Interior point algorithms}
Interior point algorithms solve \eqref{e:ds} by solving problems of the form
\begin{equation} \label{e:penalty}
\min_{S \in DS} E_{\alpha} \fnorm{AS-SA}^2+\alpha F(S)
\end{equation}
and taking $\alpha \rightarrow 0 $ to obtain a solution for \eqref{e:ds}. The function $F$ is chosen so that it explodes at the boundary, and so the constraints $S \geq 0 $ will never be active in \eqref{e:penalty}. A common choice \cite{wright1999numerical} for $F$ is $F(S)=-\sum_{ij} \log(S_{ij}) $. Specialized solvers for \eqref{e:ds} such as \cite{rangarajan1996novel,Justin} often use $F(S)=\sum_{ij}S_{ij}\log S_{ij} $. Note that while this $F$ does not explode at the boundary, its derivatives do.  

To include both choices of $F$, is well as other possible choices, we will deal with general $F$ which are of the form
$$F(S)=\sum_{ij}f(S_{ij}) $$
where $f:\RR_{\geq 0} \to \RR \cup \{\infty\}$ is continuous and strictly convex and $f(t)<\infty $ if $t>0 $.  

\begin{theorem}
	Let $G$ be a symmetry group, and $F$ be a function satisfying the conditions described previously. Then for almost every $A \in \A(G) $
	 the unique minimizers $S^*_\alpha $ of \eqref{e:penalty}
		converge to $\Sc$ as $\alpha $ tends to zero.
	\end{theorem}  
\begin{proof}
We assume that $\DS(A)$ is weakly exact. This assumption holds for almost every $A \in \A(G) $. By passing to a subsequence we can assume that $S_\alpha^*$ converges to some $S^*$ in the compact set $\DS$. We need to show that $S^*=\Sc $.

We note that for any $S \in \DS,P \in G,\alpha > 0 $
$$E_\alpha(SP)=E_\alpha(S)=E_\alpha(PS). $$
Since $S_\alpha^*$ is the unique minimizer of $E_\alpha$ this equality implies that $S_\alpha^*$ is invariant under multiplication by elements of $G$ from the right and the left. Thus this is true for $S^*$ as well. Due to continuity $F $ is bounded from below and so it can be shown that 
$$\fnorm{AS^*-S^*A}^2=\lim_{\alpha \rightarrow 0} \fnorm{AS_\alpha^*-S_\alpha^*A}^2=0.$$
It follows that $S^* $ is a convex automorphism and since $\DS(A)$ is weakly exact $S^* \in \N(G) $. Since we also showed $S^*$ to be invariant under multiplication by $G $ from the left and right it follows that $S^*=\Sc $.

\end{proof}

\section{Retrieving isomorphisms}\label{sec:iso}

\begin{wrapfigure}[18]{R}{0.4\textwidth}\vspace{-0.4cm}
	\centering
	\includegraphics[width=0.4\textwidth]{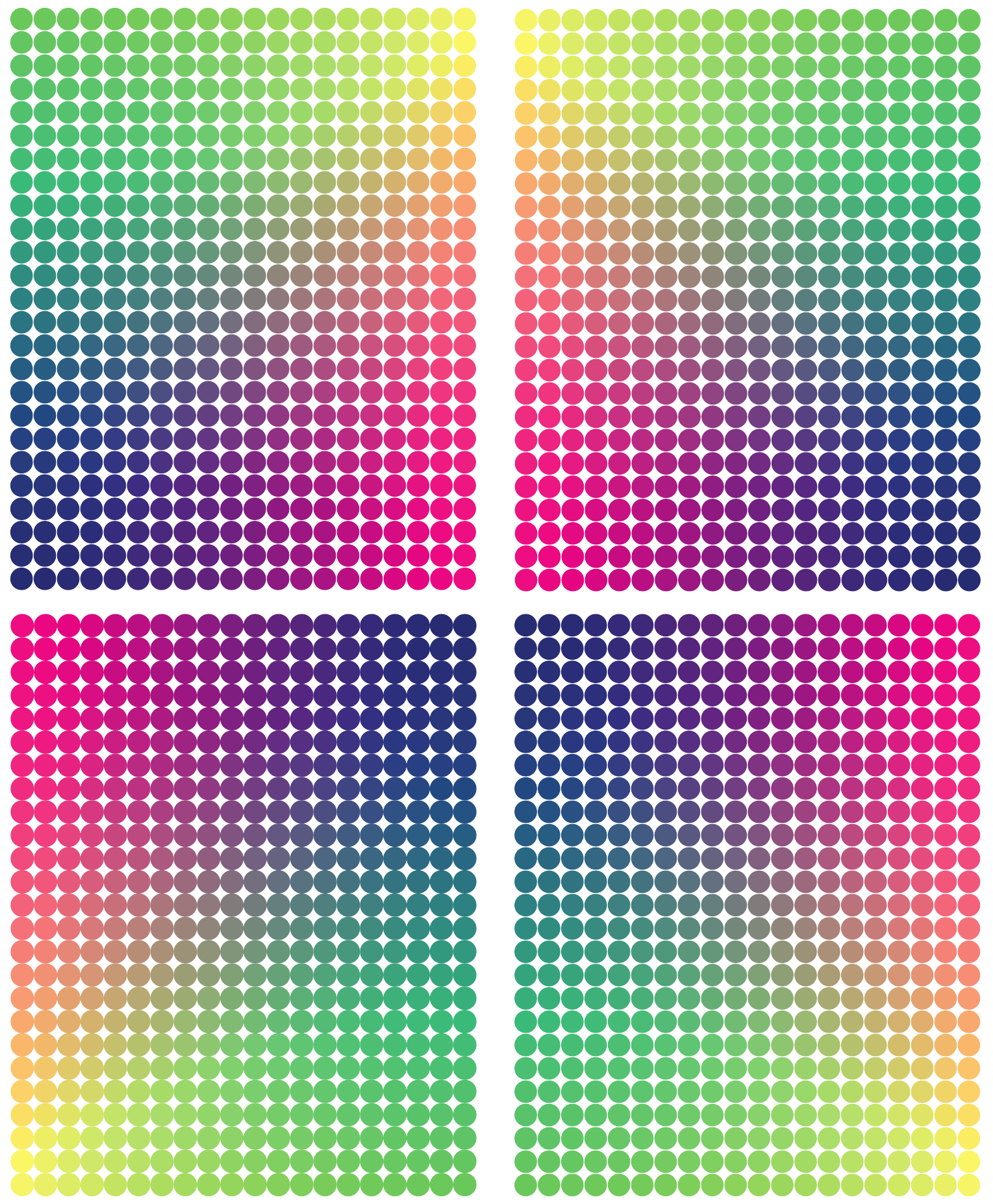}
	\caption{\small Efficeint retrieval of all symmetries of a grid. }\label{fig:rectangle}
\end{wrapfigure}
In this section we discuss how convex exactness can be used to retrieve isomorphisms. We will discuss two classes of methods. The first class searches for extreme points of the convex set of convex isomorphisms. We will show that under the assumptions of  Proposition~\ref{prop:deterministic} all isomorphisms of the graphs can be retrieved quite efficiently. However finding extreme points is not a stable methods for retrieving isomorphisms once noise is introduced. This leads to the second class of methods, which we call projection methods. Projection methods are the methods typically used in practice to achieve a permutation solution from the original solution of the DS relaxation. We show theoretically that the popular "convex to concave" projection method is able to retrieve a correct isomorphism, and explore experimentally the behavior of this method as well as the $L_2$ projection method when noise is introduced.  

\subsection{Finding extreme points}
When convex exactness holds, finding an isomorphism is reduced to the problem of finding an extreme point of the optimal set $\CIso(A,B) $ defined by the linear constraints
\begin{align*}
AS=SB\\
S \in \DS.
\end{align*}
An extreme point(=basic feasible solution) of this linear feasibility program can be found using the simplex algorithm. Extreme points can also be found using interior point algorithms by optimizing a random linear energy over $\CIso(A,B) $. In \cite{PMexact} a similar problem is discussed, and it is shown that if the linear energies are randomly drawn from the uniform distribution on $S^{n^2-1} $, then with probability one the obtained linear program will have a unique solution, which will be an extreme point. Moreover  all extreme points will be obtained with equal probability.

\begin{wraptable}[11]{R}{0.3\textwidth}
	\centering
	\includegraphics[width=0.3\textwidth]{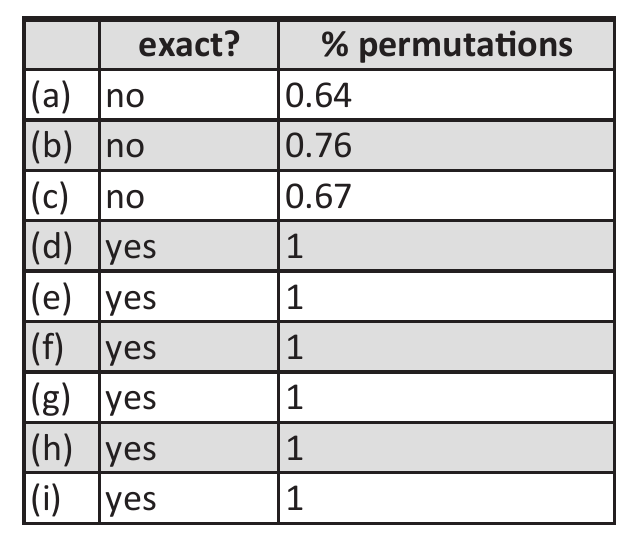}
	\vspace{-0.7 cm}
	\caption{}
	\label{tab:failures}
	\vspace{0 cm}
\end{wraptable}
Table~\ref{tab:failures} shows the successfulness of the latter method in returning isomorphisms for symmetric problems in which convex exactness holds. For each of the nine symmetry groups $G_i $ defined by the shapes in Figure~\ref{fig:failures} we generated $100$ random graphs in $\V(G_i)$ according to the distribution $\mu_{G_i}$. For each such graph we then found an extreme point by maximizing a random linear energy over the optimal set. As shown in Table~\ref{tab:failures} for the graphs $G_i, i>3 $ for which convex exactness holds generically, this algorithm succeeded in returning a permutation in all $100$ experiments. For the groups $G_i, i=1,2,3 $ for which convex exactness does not hold, this algorithm returned permutations in more than half of the experiments, but non-integer solutions were also obtained. This is due to the fact that the optimal set contains non-integer extreme points in this case.

Next we suggest a more efficient method for obtaining all extreme points of the set of convex isomorphisms, under the assumption that the assumptions of Proposition~\ref{prop:deterministic} hold and $k=|\Iso(A)| $ is not too large $k<<n $.

 If $s(A) $ is discriminative, then the centroid solution $\Sc $ can  be found directly as described in the previous section.

Once a convex isomorphism $S=\Sc$ was found, we use the technique of \cite{pataki1996cone} to find an extreme point. We now describe this technique:

We begin with some preliminaries:
For $S,T \in \RR^{n \times n} $, we say that $S \preceq T $ if $S_{ij}=0 $ whenever $T_{ij}=0 $. We say that $S \prec T $ if $S \preceq T $ but the converse inequality $T \preceq S$ does not hold.

 A \emph{face} of a convex set $K$ is a subset $F \subseteq K $ such that for all $ x,y \in K $ and $t \in (0,1) $ satisfying
$$(1-t)x+ty \in F $$ 
necessarily $x,y \in F $. An \emph{extreme point} is a face which is a singleton. If $K $ is a convex compact set then it is the convex hull of its extreme points $E$. Moreover, for each face $F \subseteq K$,
\begin{equation}\label{e:faceExtreme}
F=\conv(E \cap F).
\end{equation}

Every $S \in \CIso(A,B)$ defines a face
$$F(S)=\{Q \in \CIso(A,B)| \quad Q \preceq S \} $$
and an affine space obtained from $F(S) $ by removing the positivity constraints, i.e.,
$$V(S)=\{R| \quad R \one =\one, \quad R^T \one =\one, AR=RB, R \preceq S \}. $$
We note that $S $ is in the relative interior of $F(S) \subseteq V(S) $. This means that for all $R \in V(S) $ there is a sufficiently small $t>0 $ such that $(1-t)S+tR \in F(S) $. 
The boundary of $F(S) $ in $V(S) $ is the set: 
$$\partial F(S)=\{Q \in F(S)| \quad Q \prec S\}. $$
We can now describe the algorithm of \cite{pataki1996cone}: 

\begin{enumerate}
\item We are given as input some $S \in \CIso(A,B) $ and set $r=0 $ and $S_r=S$.
\item We compute a spanning subset to the affine space $V(S_r) $. If $V(S_r)=\{S_r \} $ then $S_r$ is an extreme point and we are done. 
\item Otherwise  we  choose some $R_r \neq S_r $ in $V(S_r) $. We then find the unique $t>0 $ such that 
$$(1-t)S_r+tR_r \in \partial F(S) $$
and set $S_{r+1} $ to be the matrix on the left hand side. We then return to the previous step.
\end{enumerate}
 The iterative process can only terminate when $V(S_r)=\{S_r\}$. This will necessarily occur after a finite number of steps since
$S_{r+1} $ always has more zeros than $S_r $. In the convex exact case, a permutation will be attained within  $k=|\Iso(A,B)| $ steps. This is because each face  $F(S_{r+1})$ is strictly contained in the former face $F(S_r) $ and therefore according to \eqref{e:faceExtreme} the number of extreme points=permutations in $F(S_{r+1}) $ is strictly smaller than the number of extreme points in $F(S_r) $. 

Once a permutation $P(1) $ is obtained, an additional permutation can be sought for by repeating the process above, but beginning with $S_0^1=(1-t)S_0+tP(1)$ where $t<0$ is the smallest possible so that $S_0(t) $ is doubly stochastic. This choice gives an initial convex isomorphism such that $P(1) \not \in F(S_0) $, guaranteeing that the algorithm will return a new permutation $P(2)$. In the next step we can set $S_0^2=(1-t)S_0^1+tP(2) $ and continue in this manner until we obtain a collection of isomorphisms $P(1),\ldots,P(L) $, and $\Sc $ is a convex combination of these isomorphisms. In fact under the full orbit assumption $P(1),\ldots, P(L) $   will be all the isomorphisms. This is because $\Sc $  can be written as a positive convex combination of all members of $\Iso(A,B) $, and the members of $\Iso(A,B) $ are linearly independent, implying that this is the only possible convex combination giving $S$, so that that all isomorphisms were obtained. The linear independence of $\Iso(A,B)$ follows from the fact that it has  full orbit, and so each isomorphism has a non-zero coordinate $i,j $ on which all other isomorphisms vanish.

From a computational perspective, under the conditions of Proposition~\ref{prop:deterministic},  The algorithm above will return an isomorphism within $k$ steps, and all isomorphisms within $O(k^2) $ iterations. Computing the first affine space $V(\Sc) $ is basically the problem of finding a linear basis to the solution set of the linear equations defining $V(\Sc)$. Since $\Sc $ has at most $nk$ non-zero entries, this is a linear equation in $O(n)$ variables instead of $n^2$ variables. For finding the subsequent affine spaces $V(S_r) $ additional computational saving can be obtained due to the fact that $V(S_r) \subseteq V(\Sc) $. Thus all elements in $V(S_r)$ are affine combinations of $k+1 $ spanning element of the affine space $V(\Sc)$, so $V(S_r) $ is obtained by solving a linear equation in only $k+1 $ variables.

Figure~\ref{fig:rectangle} shows the results of applying the algorithm described above to find the symmetries of a $20 \times 25 $ grid. The grid has a reflective symmetry group with full orbit and thus fulfills the conditions of Theorem~\ref{th:main}. We took $A=B$ to be the Euclidean distance matrix of the grid (here $n=500 $) and used the algorithm described above to obtain all symmetries of the grid. In our implementation in Matlab this calculation took around ten seconds.

\subsection{Projection methods}
The classical approach \cite{Aflalo}  for projecting a permutation solution from the doubly stochastic relaxation is using the standard $L_2$ projection, which can be implemented as a linear program and solved efficiently using the Hungarian algorithm. See \cite{zaslavskiy2009path} for more details. A more accurate and more computationally demanding method is the "convex to concave" method. We will explain this method in the formulation used in the \DSpp~algorithm \cite{DSpp}. Similar suggestions appear in \cite{zaslavskiy2009path,ogier1990neural}. We then prove \DSpp~obtains a permutation solution in the convex exact case (up to some technicalities which will be explained), and  examine the behavior of both projection methods when noise is added.

\paragraph{Convex to concave projection}
The convex to concave method sequentially solves optimization problems of the form
\begin{equation}\label{e:dspp}
\min_{S \in DS}E(S,a)=\norm{AS-SB}_F^2+a(n-\fnorm{S}^2), \quad a \geq 0.
\end{equation}
The strictly concave function
$$g(S)=n-\fnorm{S}^2 $$
is non-negative on DS, and $g(S)=0 $ if and only if $S$ is a permutation. Additionally if $a$ is sufficiently large so that $E(S,a)$ is strictly concave, then the (global and local) minima of \eqref{e:dspp} will necessarily be  permutations since the minima of a strictly concave function on a convex compact set are always extreme points. Thus the global minimum of the relaxed $\eqref{e:dspp} $ and the original quadratic assignment problem are identical. Note however that since \eqref{e:dspp} is not convex computing the global minimum is no longer tractable. 

Building on this observation, the convex to concave method minimizes (locally) a sequence of optimization problems of the form \eqref{e:dspp} on a sequence of choices
$$a_0<a_1< \ldots <a_N $$
and in each step uses the obtained solution $S_i $ as a warm start to the optimization of $E(S,a_{i+1}) $. The first point $a_0 $ is selected so that $E(S,a_0) $ is convex, and the last point is selected so that $E(S,a_N) $ is strictly concave and thus the obtained local minima $S_N $ is guaranteed to be a permutation.

The first point $a_0 $ can be selected to be zero to ensure that $E(S,a_0) $ is convex. However a better selection is $a_0=\lambdaMin $ where $\lambdaMin\geq 0$ is the minimal eigenvalue of the quadratic form
$$S \mapsto \fnorm{AS-SB}^2 $$
when restricted to the subspace
$$\{S| \quad S \one =0, \quad  S^T \one =0 \}. $$
Similarly the last point $a_N $ is selected to be (slightly larger than) the maximal eigenvalue $\lambdaMax$ of the same quadratic form over the same subspace. This choice ensures that $E(S,a_N) $ is (strictly) concave. The remaining points $a_i$ can be uniformly sampled in the interval $[a_0,a_N] $ (for lack of a better strategy). 

Note that if $A $ and $B$ are isomorphic, then for any $a>0$ the global minimizers of $E(S,a) $ are precisely $\Iso(A,B) $ (while for $a=0$ the global minimizers are $\CIso(A,B) $ ). This observation suggests the "convex to concave" method  may be successful in retrieving isomorphisms even for symmetric problems, and possibly could return integer solutions $S_i $ even for $i<N$. We now give a theoretical justification for these observations.

We assume that we obtain each $S_i^*$ from a local minimization algorithm with the following properties:
\begin{enumerate}
	\item Monotonicity: $E(S_i^*,a_i) \leq E(S_{i-1}^*,a_i)  $.
	\item The first-order necessary condition (KKT conditions) for local minima is satisfied at $S_i^* $.
	\item The second-order necessary condition for local minima of $E(\cdot,a_i)$ is satisfied at $S_i^* $. That is
	\begin{equation}\label{e:Hess}
	\vec(S-T)^T H \vec(S-T) \geq 0, \text{ for all } S,T \in F(S_i^*).
	\end{equation}
	Here $H$ is the Hessian of the quadratic form $E(\cdot,a_i) $.
\end{enumerate}
Under these assumptions we prove
\begin{theorem}\label{th:DSpp}
	Assume $A $ and $B $ are isomorphic and the DS relaxation is convex exact at $A$. Assume $S_i^*, i=0,\ldots,N $ satisfy conditions (1)-(3).
	Then if $a_1 $ is sufficiently close to $a_0$
	$$S_i^* \in \Iso(A,B), \text{ for all } i\geq 1. $$
\end{theorem}
The theorem is proved in Appendix~\ref{A:convex2concave}.

\paragraph{Isomorphism retrieval for noisy problems}
We examine the behavior of the DS relaxation coupled with the projections described above for noisy symmetric problems 
by conducting the following experiment:

We construct a random bilaterally symmetric graphs $A \in \RR^{n \times n} $ and choose $B=A$. We then perturb these graphs by two randomly selected symmetric matrices $\Delta A, \Delta B $, and solve the \DS~relaxation using both projection methods. We do this for $n=10,30,50$ and for matrices  $\Delta A, \Delta B $ with Frobenius norm $\epsilon=10^\alpha $ where we use ten values of $\alpha$ uniformly chosen from the interval $[-3,0]$. The graph $A$ is  chosen by computing an isometry $L:\RR^k \to \V(G) $ where $G$ is a permutation subgroup with two elements, and then sampling a vector $x \in \RR^k $ uniformly from the unit sphere to obtain $A=L(x) $. For each fixed value of $n,\alpha$ we repeat $100$ different instances of the experiment, and compute the retrieval ratio of both methods, which we define as the number of times the method returned a permutation from $G$ divided by the number of experiments (100). The results are shown in Figure~\ref{fig:noise}. 

It can be seen that both methods succeed in retrieving a correct permutation at low noise levels, but the convex-to-concave method (denoted by \DSpp) is more successful than the $L_2$ projection method (denoted by \DS) at higher noise levels. In the case $n=10$ we also add the "groud truth retrieval ratio", that is the number of instances in which the global minimizer of the graph matching energy was indeed in $G$ divided by the number of experiments. It can be seen that as the noise level approaches $10^0=1 $ the noise "takes over the problem" and the members of $G$ are no longer the global minimizers. The ground truth solutions was obtained by the semi-definite relaxation of \cite{Itay} which is known to be very tight, though computationally expensive. We verify that the solution obtained from the semi-definite relaxation is indeed the correct solution by checking that the difference between the lower bound provided by the relaxation and the upper bound provided by projecting the solution of the relaxation are negligible.

As a side note, we observe that at low noise levels \DSpp~obtains a solution in $G$ after two iterations in accordance with Theorem~\ref{th:DSpp}, and that even at high noise levels a permutation solution is usually attained after four iterations. This indicates that it might be worthwhile to choose a smaller $a_N$, or alternatively to consider less steps in the convex-to-concave process.
\begin{figure}[t]
	\centering
	\includegraphics[width=\columnwidth]{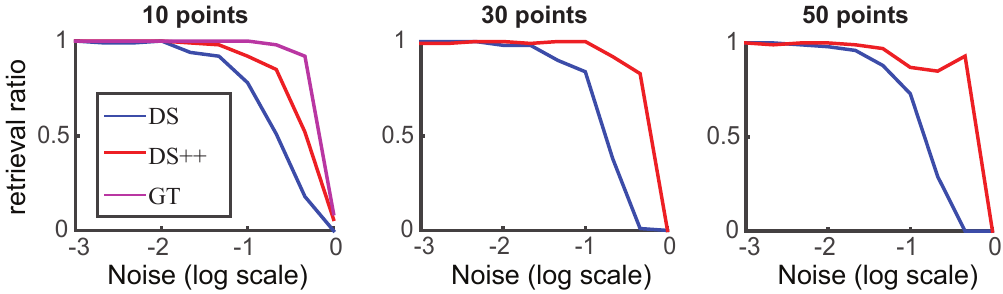}
	\vspace{-0.5cm}
	\caption{\small Evaluation of \DS~ and \DSpp~ accuracy as a function of noise for symmetric graphs. The details are explained in the text.}
	\label{fig:noise}
	\vspace{-0.1 cm}
\end{figure}

\bibliographystyle{apalike}
\bibliography{DSbibfile}

\appendix
\vspace{-5pt}

\section{Convex to concave} \label{A:convex2concave}

\begin{proof}[Proof of Theorem~\ref{th:DSpp}]

	We begin with some preliminaries. First note that if $S_{i-1}^* $ is an isomorphism, then $E(S_i^*,a_i)=0 $ due to the monotonicity condition and thus $S_i^* $ is an isomorphism.	
	
	In the asymmetric case the claim is trivial: Since $E(S,a_0) $ is convex its local minimizers are also global minimizers. Since in the asymmetric case is the unique isomorphism between $A$ and $B$ is the only global minimizer for any $a_0 \geq 0 $ it follows that $S_0^* $ is that unique minimizer. Therefore in this proof we will focus on the symmetric case only. 	
	
	In the symmetric case there are at least two isomorphisms $P_0,P_1 $. Thus $P_1-P_0$ is an eigenvector of the energy $E(S) $ with eigenvalue $\lambdaMin=0$ and so $a_0=0 $. 
	
	Our claim follows easily from the following lemma:
	\begin{lemma}\label{lem:U}
		There exists an open set $U$ containing $\CIso(A,B) $ such that for all $a>0$, The only points satisfying the first and second order conditions for local minimization of $E(S,a) $ are the members of $\Iso(A,B) $.
	\end{lemma}
	To obtain the theorem from the lemma, let $m>0 $ be the minimum of $E(S) $ on the compact set $DS \setminus U $. For any $a_1>0$  sufficiently small so that $a_1 \max_{S \in DS} g(S)<m $ we obtain
	$$E(S_1^*,a_1)\leq E(S_0^*,a_1)=a_1 g(S_0^*)<m \leq \min_{S \in DS \setminus U} E(S,a_1).  $$
	It follows that  $S_1^* \in U $ and since it satisfies the first and second order conditions for local minimization of $E(\cdot,a_1)$ it follows that  $S_1^* \in \Iso(A,B) $.
	
	\paragraph*{Proof of Lemma~\ref{lem:U}.} We construct for each $S \in \CIso(A,B) $ an open set $U_S $ satisfying the properties required from $U$ and then choose 
	$$U =\cup_{S \in \CIso(A,B)} U_S. $$
	If $S$ is a convex isomorphism but not a permutation we choose
	$$U_S=\DS \cap \{Q| \quad Q_{ij}>0 \text{ if } S_{ij}>0  \}=\{Q| \quad S \in F(Q) \}. $$
	Fix some $Q \in U_S $, we claim that the second-order necessary condition for minimizing $E(\cdot,a) $ is not satisfied at $Q$ for any $a>0$. Since $S$ is a convex combination of isomorphisms we can choose an isomorphism $P$ such that $S \succeq P $, and so $P,S \in F(Q) $. Since $P,S$ are both zeros of the convex quadratic form $E(\cdot,0) $, it follows that the second-order condition does not hold since (denoting by $H_g$ the Hessian of $g$ )
	$$\vec (P-S)^TH\vec (P-S)= a \ \vec (P-S)^TH_g\vec (P-S)=-2a  \fnorm{P-S}^2<0.$$   
	
		For isomorphisms $P$ we choose $U_P$ as follows: 
		
		For any $S \in \DS\setminus \perm_n $ the concavity of $g$ implies
		$$\nabla_P^T g(S-P) \geq g(S)-g(P)>0. $$
		In particular this is true for any $S$ in the compact set 
		$$K =\{S \in DS|  \sup_{i,j}\left|P_{ij}-S_{ij}\right|=0.5  \} .$$   
		Since $g$ is $C^1 $ the function
		$$F(Q,S)=\nabla_Q^Tg(S-Q) $$
		is continuous. Thus,  there is a neighborhood $U_P \subseteq \{S \in DS| \sup_{i,j} |P_{ij}-S_{ij} | \leq \frac{1}{4} \} $ of $P$ on which
		\begin{equation}\label{e:nabla}
		\nabla_Q^Tg(S-Q)>0, \quad  \forall(Q,S) \in U_P \times K.
		\end{equation}
		Fix some $Q \in U_P$. Define
		$$t_M=\sup\{t \geq 0|Q(t)=(1-t)P+tQ \in \DS \}.$$
		Note that $\sup_{i,j} \left|P_{ij}-Q_{ij}(t_M)\right|=1 $ and therefore there is some $1<t_0<t_M $ such that $S=Q(t_0) \in K $. It follows from \eqref{e:nabla} that 
		$$0<\nabla_Q^Tg(Q(t_0)-Q)=(1-t_0) \nabla_Q^T g (P-Q)$$
		and therefore
		$$\nabla_Q^T g(P-Q)<0 $$
		The convexity of $E$ implies that for all $Q \in U_P $,
		$$\nabla_Q^T E(P-Q) \leq E(P)-E(Q) \leq 0 .$$
		From the last two equations it follows that for any $a>0$ the energy $E(\cdot,a) $ has a descent direction $P-Q $ at any point $Q \in U_P \setminus \{P\} $. Not that this direction is orthogonal to the gradients of the constraints defining $\DS$, since $Q+t(P-Q) $ is feasible if $|t|$ is small enough. Thus the first-order condition does not hold at $Q$. 
\end{proof}


\end{document}